\documentclass[12pt]{article}

\usepackage{amsmath,amsthm,amssymb,amsfonts}

\newtheorem{theorem}{Theorem}[section]
\newtheorem{lemma}[theorem]{Lemma}
\newtheorem{proposition}[theorem]{Proposition}

\newtheorem{question}[theorem]{Question}

\numberwithin{equation}{section}

\newcommand{\R}         {\mathbb{R}}

\newcommand{\Z}         {\mathbb{Z}}
\newcommand{\N}         {\mathbb{N}}

\newcommand{\cH}         {\mathcal{H}}
\newcommand{\cE}         {\mathcal{E}}
\newcommand{\cF}         {\mathcal{F}}

\newcommand{\st}         {\mbox{ s.t. }}
\newcommand{\aand}         {\;{\mbox{ and }}\;}

\begin{document}

\title{On the Allen-Cahn
equation in the Grushin plane:\\
a monotone entire solution\\
that is not one-dimensional}

\author{\textit{Isabeau Birindelli}\ \ \& \ \ 
\textit{Enrico Valdinoci}}

\maketitle

\begin{abstract}
We consider
solutions of the
Allen-Cahn
equation in the whole Grushin plane
and we show that if they are monotone
in the vertical direction, then they are stable
and they satisfy a good energy estimate.

However, they are not necessarily
one-dimensional, as a counter-example shows. 
\end{abstract}

\section{Introduction}

We consider here the
Grushin plane $G$ (see~\cite{Grushin}),
that is $\R^2$ endowed 
with
the vector fields $X=\frac{\partial}{\partial x}$ and $Y
=x\frac{\partial}{\partial y}$.
We also define $T:=[X,Y]=\frac{\partial}{\partial y}$.

The Grushin gradient is then $\nabla_G:=(X,Y)$
(with the coordinates taken in the $(X,Y)$-frame)
and the Grushin Laplacian is $\Delta_G:=X^2+Y^2$.

We denote by $<\cdot,\cdot>$ the standard
scalar product
(when the vectors are taken in the $(X,Y)$-frame),
so that, for a smooth function $v$ we have
$$ |\nabla_G v(\zeta)| = \sqrt{<\nabla_G v(\zeta),
\nabla_G v(\zeta)>}=\sqrt{
(Xv(\zeta))^2+(Yv(\zeta))^2
},$$
for any $\zeta\in\R^2$.

Moreover, the Grushin norm
on $G$ is defined as
$$
\| (x,y)\|:=\sqrt[4]{|x|^4+4|y|^2}
$$
for any $(x,y)\in\R^2$, and then the Grushin ball 
of radius $R> 0$ centered at $\zeta \in\R^2$ is
$$ B_R (\zeta):= \{ \eta\in\R^2
\st \|\eta-\zeta\|<R\}.$$

The main purpose of this paper is to
study
solutions of the
Allen-Cahn
equation in the Grushin plane, that is
\begin{equation}\label{main pde}
\Delta_G u(\zeta) + f(u(\zeta))=0\end{equation}
for any $\zeta\in\R^2$.

We take, for simplicity, $f\in C^1$,
though less regularity is also possible to be dealt with.

A particular case of interest
is when $f=-W'$ and $W$ is a double-well potential.
Namely, through this paper, 
we denote by $W$ a function with the following properties:
$W\in C^{2} (\R)$ is
an even function for which $W(\pm 1)=0\leq W(r)$
for any $r\in\R$, $W''(0)\ne 0$, $W''(\pm 1)\ne0$,
and such that
\begin{equation}\label{UN}
{\mbox{
$W'(s)=0$ if and only if $s\in \{-1,0,+1\}$.
}}
\end{equation}
\bigskip

Inspired by a famous conjecture of De Giorgi (see
\cite{DG}), one may wonder under which conditions
the solutions of \eqref{main pde}
are one-dimensional, i.e., 
their level sets are straight lines and so,
up
to rotation, they depend on only one variable (at least when
$f=-W'$).

Natural requirements for such symmetry
are monotonicity and stability
conditions.
Namely,
if $u$ is a solution of \eqref{main pde},
we say that $u$ is stable if
\begin{equation}\label{Stab}
\int_{\R^2} |\nabla_G \phi(x)|^2-\int_{
{\mathcal{G}}_u}
f'(u(x)) \big(\phi(x)\big)^2
\,dx\ge0
\end{equation}
for any $\phi\in C^\infty_0(\R^2)$.

Stability is a natural condition in the calculus
of variation, since it states that the energy
functional associated to \eqref{main pde} has
non-negative second derivative. The stability condition
has thus been widely used in connection with the
problems posed by \cite{DG} (see, for instance, \cite{AAC, FSV}
and references therein).

Also, in the Euclidean setting, the stability condition
holds true whenever $u$ is monotone in some direction.
The analogy in the Grushin setting is somehow more
delicate, since the space is not homogeneous with respect
to the choice of a particular direction.

Thus, the monotonicity studied in this paper is the following.
We are mostly concerned with solutions that are monotone
in the $y$-direction, that is for which
\begin{equation}\label{Mo}
T u(\zeta)>0 \,{\mbox{for any}}\,
\zeta\in\R^2.
\end{equation}
We shall show that \eqref{Mo} implies~\eqref{Stab}
(see Proposition \ref{677888188} below).
\bigskip

Symmetry properties for
solutions of \eqref{main pde}
in the Grushin plane
have been recently studied
in \cite{VF}.
For instance, \cite{VF} pointed out the following result:

\begin{theorem}
Let us assume that $u$ is a stable solution of
\eqref{main pde}
in the whole $\R^2$ such that
$$ TYuXu(\zeta)-
TXuYu(\zeta)\leq 0\qquad{\mbox{for any $\zeta
\in\R^2$}}.
$$
Suppose that there exists $C_o\geq 1$ in such a way that
\begin{equation}\label{enesty}
\int_{B(0,R)}x^2| \nabla_G u|^2 \leq C_o R^4,
\end{equation}
for any $R\ge C_o$.

Assume also that
\begin{equation}\label{56789999900}
\nabla_G u (\zeta)\ne 0\,{\mbox{ for any }}\,\zeta\in\R^2.
\end{equation}
Then, $u$ depends only on the $x$-variable.
\end{theorem}

We observe that \eqref{56789999900}
is also a
sort of monotonicity condition, while
\eqref{enesty} is an energy growth requirement
(and energy bounds are often needed in the Euclidean
case too, see \cite{AAC}).
We shall show that \eqref{Mo} implies also
\eqref{enesty}, at least when $f=-W'$
(see Theorem \ref{EB} below). \bigskip

This said, a natural question arises.
Namely, \begin{question}\label{QQQ}
Is it true that bounded solutions
of \eqref{main pde} which satisfy \eqref{Mo}
are one-dimensional (at least for $f=-W'$)?
\end{question}

Note that one may be quite tempted to answer yes to such
a question, since \eqref{Mo} implies both
the stability condition and the good energy growth
in \eqref{enesty} (again, see for this
Proposition \ref{677888188}
and Theorem \ref{EB} here below).

The main purpose of this paper is in fact to show
that the above question has a \emph{negative
answer}.

This will be accomplished in Theorem \ref{MMM},
by constructing a counter-example which follows the
lines of the one in \cite{BL}.
\bigskip

The paper is organized in the following way.
After gathering some elementary observations in
Section \ref{S1}, we 
point out in Section \ref{S0} that the monotonicity
condition in \eqref{Mo} implies the stability condition
in \eqref{Stab}.

Then, we develop in Section \ref{S2}
the energy estimates which show 
that the monotonicity
condition in \eqref{Mo} also implies
the energy growth in \eqref{enesty}.

Finally, Section \ref{S3} contains the construction
of the counter-example which shows that Question \ref{QQQ}
has a negative answer.

\section{Preliminaries}\label{S1}

We collect in this section some elementary, but
useful, observations. The expert reader may surely
skip this section.
 
\subsection{An integration by parts}

We now point out a variation of Green formula,
complicated here by the non-ho\-mo\-ge\-ne\-ous
Grushin scaling.

\begin{lemma}\label{Green}
Let $u\in \Lambda^2(\R^2)$, $v\in
\Lambda^1(\R^2)$. Suppose that $|\nabla_G u|\in
L^\infty(\R^2)$ and that
$v\ge 0$. 

Then
\begin{equation}\label{1ai889}
\begin{split}
& \int_{B_R(0)}\Big(<\nabla_G u,\nabla_G v>
+\Delta_G u v\Big)\\& \qquad\qquad \geq
-R^2\, \| \nabla_G u\|_{L^\infty(\R^2)}
\int_{\partial B_1(0)} v(RX, R^2Y)\,d\cH^1(X,Y).
\end{split}\end{equation}
\end{lemma}

\begin{proof} We set $U(X,Y):=u(RX,R^2Y)$,
$V(X,Y):=v(RX,R^2Y)$.
Then, by changing variable, we have
\begin{eqnarray*}
&& \int_{B_R(0)}\Big(<\nabla_G u,\nabla_G v>
+\Delta_G u v\Big)\\
&=&
\int_{B_R(0)}\Big( \partial_x u(x,y)\partial_x v(x,y)
+x^2
\partial_y u(x,y)\partial_y v(x,y)
\\ &&
+ \partial_{xx}u(x,y) v(x,y)
+x^2 \partial_{yy}u(x,y) v(x,y)
\Big)\,dx\,dy\\
&=&\frac{1}{R^2}
\int_{B_R(0)}\Big( \partial_X U
\Big(\frac{x}{R},\frac{y}{R^2}\Big)
\partial_X V
\Big(\frac{x}{R},\frac{y}{R^2}\Big)
\\ &&+
\frac{x^2}{R^2}
\partial_Y U
\Big(\frac{x}{R},\frac{y}{R^2}\Big)
\partial_Y V
\Big(\frac{x}{R},\frac{y}{R^2}\Big)
+ \partial_{XX}U\Big(\frac{x}{R},\frac{y}{R^2}\Big)
V\Big(\frac{x}{R},\frac{y}{R^2}\Big)
\\&&+\frac{x^2}{R^2} \partial_{YY}U
\Big(\frac{x}{R},\frac{y}{R^2}\Big)V
\Big(\frac{x}{R},\frac{y}{R^2}\Big)
\Big)\,dx\,dy
\\ &=& R\int_{B_1(0)}
\partial_X U(X,Y)\partial_X V(X,Y)+X^2
\partial_Y U(X,Y)\partial_Y V(X,Y)\\&&+
\partial_{XX} U(X,Y)V(X,Y)+
X^2 
\partial_{YY} U(X,Y) V(X,Y)\,dX\,dY.
\end{eqnarray*}
Thence, by the standard Euclidean Divergence Theorem,
\begin{eqnarray}\label{8uauiiiq}
\nonumber
&& \int_{B_R(0)}\Big(<\nabla_G u,\nabla_G v>
+\Delta_G u v\Big)\\
&=& R \int_{B_1(0)}
{\rm div}\, \Big[V(X,Y) \Big( 
\partial_X U(X,Y),
X^2 \partial_Y U(X,Y)\Big) \Big]
\,dX\,dY\\
&=& R
\int_{\partial B_1(0)}
V(X,Y) \Big( 
\partial_X U(X,Y),
X^2 \partial_Y U(X,Y)\Big) \cdot
\nu^E (X,Y)
\,d\cH^1(X,Y),\nonumber
\end{eqnarray}
where 
``$\cdot$'' denotes the standard Euclidean scalar product
and ``$\nu^E$'' is the standard Euclidean outward normal
of $\partial B_1(0)$.

We write \eqref{8uauiiiq} as
\begin{eqnarray*}
&& \int_{B_R(0)}\Big(<\nabla_G u,\nabla_G v>
+\Delta_G u v\Big)\\
&=& R^2
\int_{\partial B_1(0)}
v(RX,R^2 Y) \Big( 
\partial_x u(R X,R^2 Y),
R X^2 \partial_y u(RX,R^2Y)\Big)\\&&
\qquad\qquad \cdot
\nu^E (X,Y)
\,d\cH^1(X,Y)
\end{eqnarray*}
and so, since the Euclidean norm of $\nu^E$ is $1$,
\begin{equation}\label{ai889}
\begin{split}
&\int_{B_R(0)}\Big(<\nabla_G u,\nabla_G v>
+\Delta_G u v\Big)\\
&\qquad \geq - R^2
\int_{\partial B_1(0)}
v(RX,R^2 Y) \Big| \Big(
\partial_x u(R X,R^2 Y),
R X^2 \partial_y u(RX,R^2Y)\Big)
\Big|_E
\,d\cH^1(X,Y),
\end{split}\end{equation}
where $|\cdot|_E$ is the Euclidean norm.

We now observe that, for any $(X,Y)\in\partial B_1(0)$,
we have $|X|\le 1$ and
\begin{eqnarray*}
&& \Big| \Big(
\partial_x u(R X,R^2 Y),
R X^2 \partial_y u(RX,R^2Y)\Big)\Big|_E^2
\\ &=&
\Big( 
\partial_x u(R X,R^2 Y)\Big)^2+
R^2 X^4 \Big(\partial_y u(RX,R^2Y)\Big)^2
\\ &\le&
\Big( 
\partial_x u(R X,R^2 Y)\Big)^2+
R^2 X^2 \Big(\partial_y u(RX,R^2Y)\Big)^2
\\ &=&
|\nabla_G u (RX,R^2 Y)|\\ &\leq&
\| \nabla_G u\|_{L^\infty(\R^2)}
.
\end{eqnarray*}
{F}rom this and \eqref{ai889} we get \eqref{1ai889}.
\end{proof}

\subsection{An interpolation inequality}

We point out the following elementary estimate:

\begin{lemma}
Let $h\in C^2 (\R)$.
Then,
\begin{equation}\label{Interpo}
\| h' \|_{L^\infty (\R)}\le
2\Big(\| h \|_{L^\infty (\R)}+
\| h'' \|_{L^\infty (\R)} \Big).\end{equation}
\end{lemma}

\begin{proof} We may assume that both~$\| h \|_{L^\infty (\R)}$
and~$\| 
h'' \|_{L^\infty (\R)}$ are finite, otherwise~\eqref{Interpo}
is void.

First, we observe that, for any~$j\in\Z$,
there exists~$t_j \in [j,j+1]$ in such a way that
\begin{equation}\label{st1}
|h'(t_j)|\le 2\| h \|_{L^\infty (\R)}.
\end{equation}
Indeed, if~\eqref{st1} were false, there would exist~$j_o
\in\Z$ such that~$|h'(t)|> 2\| h \|_{L^\infty (\R)}$
for any~$t\in [j_o, j_o+1]$. Since~$h'$
is continuous, this means that either~$h'(t)>
2\| h \|_{L^\infty 
(\R)}$ or~$
h'(t)<- 2\| h \|_{L^\infty (\R)}$
for any~$t\in [j_o,j_o+1]$. We assume that
the second possibility holds (the first case
is analogous). Then,
\begin{eqnarray*}&&
-2\| h \|_{L^\infty (\R)}
\leq h(j_o+1)-h(j_o)=
\int_{j_o}^{j_o+1} h'(t)\,dt \\&&
\qquad <
\int_{j_o}^{j_o+1} 
\big(-
2\| h \|_{L^\infty (\R)}\big)
\,dt=-2\| h \|_{L^\infty (\R)}.
\end{eqnarray*}
This contradiction proves~\eqref{st1}.

Then, making use of~\eqref{st1},
given any~$j\in\Z$ and any~$t\in [j,j+1]$,
\begin{equation*}\begin{split}
&|h'(t)|\le |h'(t_j)|+\left| \int_{t_j}^t
h''(s)\,ds\right|\\
&\qquad\le
2\| h \|_{L^\infty (\R)}+
\| h '\|_{L^\infty (\R)} |t-t_j|\le
2\| h \|_{L^\infty (\R)}+
\| h' \|_{L^\infty (\R)}.
\qedhere
\end{split}\end{equation*}
\end{proof}

\subsection{ODE analysis}\label{odes}

The scope of this section is an elementary analysis
of the solutions $h\in C^2(\R)$ of
\begin{equation}\label{in2} 
h''(t)=W'(h(t)) \qquad\qquad{\mbox{for
any $t\in\R$.}}\end{equation}

Recall that for any any $C^2$ solution of (\ref{in2}) and any
 $s$, $t\in\R$, 
 \begin{equation} \label{intepri} \frac{|h'(s)|^2}{2}-W(h(s))=
\frac{|h'(t)|^2}{2}-W(h(t)).
\end{equation}

Furthermore
\begin{lemma}
Let $h$ be bounded. Then,
\begin{equation}\label{17889}
\| h\|_{C^2(\R)} \leq C,\end{equation}
for a suitable $C>0$, possibly depending on $\|h\|_{
L^\infty(\R)}$.

Also, for any $t\in\R$,
\begin{equation}\label{7891}
-W\Big(\inf_\R h\Big)=-W\Big(\sup_\R h\Big)=
\frac{|h'(t)|^2}{2}-W(h(t)).
\end{equation}
\end{lemma}

\begin{proof} By construction, 
$$|h''(t)|\leq \max_{[-\| h\|_{L^\infty(\R)},
\| h\|_{L^\infty(\R)}]} |W'|$$
and so, by \eqref{Interpo},
we get \eqref{17889}.

We take 
$$ \sigma \in\Big\{ \inf_\R h,
\sup_\R h\Big\}.$$
Let also $t_n$ be a sequence for which
$$ \lim_{n\rightarrow+\infty} h(t_n)=\sigma.$$
Let $w_n(t):= h(t+t_n)$. {F}rom \eqref{17889},
we have that $w_n$ converges, up to subsequence,
in $C^1_{\rm loc}(\R)$ to some function $w\in C^1(\R)$.

We now suppose that
$\sigma = \inf_\R h$ (for this argument, the
case $\sigma=\sup_\R h$
is completely analogous). Then,
$$ w(0)=
\lim_{n\rightarrow+\infty} w_n(0)=
\lim_{n\rightarrow+\infty} h(t_n)=
\sigma \le \lim_{n\rightarrow+\infty} h(t+t_n)
= w(t)$$
for any $t\in \R$, so $w'(0)=0$ and therefore,
by
(\ref{intepri}), for any $t\in\R$,
\begin{equation*} \begin{split}
& \frac{|h'(t)|^2}{2}-W(h(t))
=\lim_{n\rightarrow+\infty}
\frac{|h'(t_n)|^2}{2}-W(h(t_n))
\\ &\quad
=\lim_{n\rightarrow+\infty}
\frac{|w'_n(0)|^2}{2}-W(w(0))
= -W(\sigma).\qedhere\end{split}\end{equation*}
\end{proof}

\begin{lemma}\label{lepa}
If $h'(t_o)=0$, then $h$ is symmetric  with respect to $t=t_o$,
that is $h(t_o-t)=h(t_o+t)$ for any $t\in\R$.
\end{lemma}

\begin{proof} We set $h_\pm (t):=h(t_o\pm t)$.
Since $h_\pm ''(t)= W'(h_\pm(t))$ for any $t\in\R$,
$h_\pm (0)=h(t_o)$ and $h_\pm'(0)=0$, we deduce
from Cauchy Uniqueness Theorem that $h_+(t)=h_-(t)$. 
\end{proof}

\begin{lemma}\label{2orm}
If $h$ has two or more critical points, then it
is periodic.
\end{lemma}

\begin{proof}
Suppose that $h'(a)=h'(b)=0$ with $b>a$ and let $T:=b-a$.
Then, utilizing Lemma \ref{lepa},
\begin{equation*}\begin{split}&
h(t+T)=h(b+(t-a))=h(b-(t-a))\\ &\qquad=h(a-(t-b))
=h(a+(t-b))=h(t-T)\end{split}\end{equation*}
for any $t\in\R$, and so $h$ has period $2T$.
\end{proof}

\begin{lemma}\label{MaMi}
If $|h|\le 1$, then
$$ \sup_\R h = -\inf_\R h.$$
\end{lemma}

\begin{proof} Let
$$ m:= \inf_\R h \aand
M:=\sup_\R h.$$
By \eqref{7891}, we have
\begin{equation}\label{rt}
W(m)=W(M).\end{equation}
Thus, by Rolle's Theorem, there exists
$\xi\in ( m, M)$ such that $h'(\xi)=0$.

{F}rom \eqref{UN}, we deduce that $\xi\in \{-1,0,+1\}$.
But since, by assumption, both $m$ and $M$
lie in $[-1,1]$, we have that $\xi\in (m,M)\subseteq
(-1,1)$ and so $\xi=0$.

This says that $m < 0< M$. Thus the claim follows
from \eqref{UN} and \eqref{rt}.
\end{proof}

\begin{lemma}
Suppose that $h$ is either non-periodic or it is constant
but not zero.
Then,
\begin{equation}\label{spe}
W\Big( \inf_\R h\Big)=
W\Big( \sup_\R h\Big)=0.\end{equation}
\end{lemma}

\begin{proof} If $h$ is constant but not zero,
then either $h$ is constantly equal to $-1$
or it is constantly equal to $+1$, because of \eqref{UN}.

Since in such cases \eqref{spe} is obvious,
we focus on the case in which $h$ is not periodic.
Then, by Lemma \ref{2orm},
\begin{equation}\label{1a}
{\mbox{$h$ has at most one
critical point.
}}
\end{equation}
In particular,
$h$ attains either its $\sup$ or its $\inf$ at either $+\infty$
or $-\infty$. So, let us assume, for definiteness
that
\begin{equation}\label{71}
\sup_\R h=
\lim_{t\rightarrow +\infty}h(t),\end{equation}
the other cases being analogous.

By \eqref{17889}, we obtain that the limit
in \eqref{71} holds in $C^1$, therefore
$$ 0= \lim_{t\rightarrow+\infty}
\int_\R h'(s+t) \phi'(s)+
W'(h(s+t))\phi(s)\,  ds=
\int_\R W'\Big( \sup_\R h\Big) \phi(s)\,ds,$$
for any $\phi\in C^\infty_0(\R)$.

This says that
$$ W'\Big( \sup_\R h\Big)=0$$
and so, by \eqref{UN},
$$ \sup_\R h\in \{-1,0,+1\}.$$
If $ \sup_\R h\in \{-1,+1\}$, then
\eqref{spe} holds true, recalling \eqref{7891}.

Thus, we consider the case
$$ \sup_\R h=0.$$
Then, recalling \eqref{1a}, we have two
possibilities: either 
\begin{equation}\label{72}
\inf_\R h=
\lim_{t\rightarrow -\infty}h(t),\end{equation}
or there exists $t_o\in\R$ such
that
\begin{equation}\label{73} h(t_o)=\inf_\R h
.\end{equation}
Now, if \eqref{72} holds, we repeat the argument
after \eqref{71} to obtain
that
$$ W'\Big( \inf_\R h\Big)=0$$
and so, by \eqref{UN},
\begin{equation}\label{a7888}
\inf_\R h\in \{-1,0,+1\}.\end{equation}
Since $h$ is not constantly equal to zero,
we have that
$$ 0=\sup_\R h> \inf_\R h$$
and so \eqref{a7888} means that 
$\inf_\R h=-1$. This implies
that
\eqref{spe} holds true, recalling \eqref{7891}.

Thus, we have only to deal
with the case in which \eqref{73} holds,
which we now show that is impossible.
Indeed, if \eqref{73}
were true, we would have $h'(t_o)=0$
and so, by \eqref{7891},
$$ W(0)=W\Big(\sup_\R h\Big)=
W(h(t_o)).$$
Thus, by \eqref{UN}, we would have that
$h(t_o)=0$. That is,
$$ \sup_\R h=0=\inf_\R h,$$
in contradiction with our assumptions.
\end{proof}

\begin{lemma}\label{0aoo}
Let $a<b$. If $h$ is monotone in $(a,b)$ then
$$ \int_a^b |h'(t)|\,dt \le 2
\sup_\R |h|.$$
\end{lemma}

\begin{proof} We have\begin{equation*}
\int_a^b |h'(t)|\,dt = \left| \int_a^b h'(t)\,dt\right|=|h(b)-h(a)|\leq
2\sup_\R |h|.\qedhere
\end{equation*}\end{proof}

\begin{lemma}
Suppose $|h|\le 1$. Then, if $h$ is not periodic, 
$$ \int_{-\infty}^{+\infty}|h'(t)|\,dt \le 4.$$
\end{lemma}

\begin{proof} By Lemma \ref{2orm}, we see
that only two cases hold: either $h'$ never vanish
or $h'$ has only one zero.

In any case, there exists $c\in \R$ in such a way that
$h'(t)\ne0$ for any $t\in (-\infty, c)\cup(c,+\infty)$.

Consequently, by Lemma \ref{0aoo}, for any $a<c<b$,
$$ \int_a^b|h'(t)|\,dt =\int_a^c |h'(t)|\,dt+
\int_c^b |h'(t)|\,dt\le
2\sup_\R |h|+2\sup_\R |h|.
$$
Then, the desired result follows by sending
$a\rightarrow-\infty$ and
$b\rightarrow+\infty$.
\end{proof}

\begin{lemma}\label{jajjjk}
Let
$$ \sigma \in \left\{ \inf_\R h, \sup_\R h\right\}.$$
Suppose $|h|\le 1$. Then,
\begin{equation}\label{aii88}
\int_{-\infty}^{+\infty} \frac{|h'(t)|^2}{2}+
W(h(t))-W(\sigma) \,dt \le C,\end{equation}
for a suitable structural constant $C>0$,
unless
$h$ is periodic and non-constant.
\end{lemma}

\begin{proof}
Since \eqref{aii88} is obvious for $h$ constant,
we focus on the case in which $h$ is not periodic.

We exploit Lemma \ref{intepri},
\eqref{17889} 
and \eqref{7891} to conclude that
\begin{equation*}
\begin{split}
&\!\!\! \int_{-\infty}^{+\infty} \frac{|h'(t)|^2}{2}+
W(h(t))-W(\sigma) \,dt \\
&=
\int_{-\infty}^{+\infty} {|h'(t)|^2} \,dt \\
& \le 
C\int_{-\infty}^{+\infty} {|h'(t)|} \,dt 
\\ &\le 4C.\qedhere
\end{split}
\end{equation*}
\end{proof}

\subsection{A compactness result}

We now point out a useful compactness criterion:

\begin{lemma}\label{CCC}
Let $u_k$ be a sequence of
solutions of \eqref{main pde} in the whole $\R^2$.
Then, up to subsequence,
$u_k$ converges locally uniformly to some
$u$ which is also a solution of
\eqref{main pde} in the whole $\R^2$.
\end{lemma}

\begin{proof}
By Grushin-elliptic regularity (see, e.g.,
\cite{Franchi, FS, K, VCSC}),
we have that
$$ \| u_k \|_{C^{\alpha} (\R^2) } \leq \bar C$$
for some $\bar C>0$,
therefore 
up to subsequence,
$u_k$ converges locally uniformly to some
$u$, and so
\begin{equation}\label{Li2}
{\mbox{$u$ is continuous.}}
\end{equation}
Moreover, for any $a\in(0,1)$, the Grushin operator
$\Delta_G$
is uniformly
elliptic in $D_a:=\{ |x|\in (a,1/a)\}$, therefore
standard elliptic estimates give that
$$ \| u_k \|_{C^{2,\beta_a} (D_a) } \leq \bar C_a$$
and so
\begin{equation}\label{88a880}
\Delta_G u (x)= W'(u(x))
\end{equation}
for any $x\in D_a$.

Since $a$ can be taken arbitrarily small,
we have that
\eqref{88a880}
holds for any $x\in\R^2\setminus\{0\}$.
But then, since the map $x\mapsto W'(u (x))$
is continuous, by means of \eqref{Li2}, it follows that
$\Delta_G u$ is continuous too and so
\eqref{88a880}
holds for any $x\in\R^2$.
\end{proof}

\subsection{Basic spectral theory}

\begin{lemma}\label{SPEC}
Fix $a\in(0,1)$, $R\ge 1$. Let~$x_o=a+R$ 
and $\Omega:= B_R (x_o,0)$.

Then, there exists~$\lambda\in \R$ for which
there exists a non-trivial solution~$\phi$
of
\begin{equation}\label{EI}
\left\{
\begin{matrix}
\Delta_G \phi+\lambda \phi=0 & {\mbox{ in $\Omega$}},\\
\phi=0 & {\mbox{ on $\partial\Omega$}}.
\end{matrix}\right.
\end{equation}
Moreover, we can take
\begin{equation}\label{EI2}
\lambda \in \Big( 0, \frac{C}{R^2}\Big],
\end{equation}
for a suitable~$C>0$.
\end{lemma}

\begin{proof}
We have
\begin{equation}\label{C0}\begin{split}
&\int_\Omega |\nabla_G v|^2\le
\int_\Omega |\partial_x v|^2+(x_o+R
)^2|\partial_y v|^2
\\&\qquad
\le (1+a+2R)^2 \int_\Omega |\nabla v|^2.
\end{split}\end{equation}
Therefore, by standard
Poincar\'e inequality,
\begin{equation}\label{C1}
\int_\Omega |\nabla_G v|^2\le C(a,R) \int_\Omega |v|^2
,\end{equation}
for any~$v\in C^\infty_0(\Omega)$.

Moreover, 
\begin{equation}\label{C2}\begin{split}
&\int_\Omega |\nabla_G v|^2\ge
\int_\Omega |\partial_x v|^2+ a |\partial_y v|^2
\\&\qquad
\ge \min\{1,a\} \int_\Omega |\nabla v|^2.
\end{split}\end{equation}

{F}rom~\eqref{C1}, we thus follow the standard
minimization argument, 
taking
\begin{equation}\label{C7}
\lambda\,=\,
\inf_{v\in C^\infty_0(\Omega)\setminus\{0\}}
\frac{
\int_\Omega |\nabla_G v|^2
}{
\int_\Omega v^2
}
\end{equation}
and we
recover compactness
from~\eqref{C2} and the classical embeddings,
proving~\eqref{EI}.

Now, if~$\lambda$ and~$\phi$ satisfy~\eqref{EI}, 
we may suppose that
$$ \int_\Omega \phi^2 = 1$$
and so
$$ \int_\Omega |\nabla_G \phi|^2 =\lambda $$
which gives that~$\lambda> 0$.

Also, from~\eqref{C7}
and a change of variable,
\begin{eqnarray*}
\lambda&=&
\inf_{v\in C^\infty_0(B_R(x_o,0))\setminus\{0\}}
\frac{
\int_{\{(x-x_o)^4+4y^2\le R^4\}} |\partial_x v(x,y)|^2+|x|^2
|\partial_y v(x,y)|^2
\,d(x,y)
}{
\int_{\{(x-x_o)^4+4y^2\le R^4\}} |v(x,y)|^2\,d(x,y)
}
\\ 
&=& 
\inf_{\psi\in C^\infty_0(B_1 )\setminus\{0\}}
\frac{
\int_{B_1} R^{-2} (\partial_w \psi(w,z))^2
+R^{-4} |x_o+R w|^2 (\partial_z \psi(w,z))^2
}{
\int_{B_1} \psi^2
}\\
&\le& 
\inf_{\psi\in C^\infty_0(B_1 )\setminus\{0\}}
\frac{
\int_{B_1} R^{-2} (\partial_w \psi(w,z))^2
+R^{-4} (3R)^2 (\partial_z \psi(w,z))^2
}{
\int_{B_1} \psi^2
}\\ &\le&
10 R^{-2}
\inf_{\psi\in C^\infty_0(B_1 )\setminus\{0\}}
\frac{
\int_{B_1} |\nabla \psi|^2
}{
\int_{B_1} \psi^2
}.
\end{eqnarray*}
This and the classical
Poincar\'e inequality imply~\eqref{EI2}.
\end{proof}

\subsection{Extension of bounded harmonic functions}

\begin{lemma}\label{EXTE}
Let~$u$ be~$\Delta_G$-harmonic in~$B_r\setminus\{0\}$.
Suppose that~$u$ is bounded in~$B_r\setminus\{0\}$.
Then, it may be extended to a $\Delta_G$-harmonic 
in~$B_r$.
\end{lemma}

\begin{proof}
The fundamental solution of~$\Delta_G$ is 
$\psi(x,y)=(x^4+4y^2)^{\frac{-1}{4}}$
(see  Theorem~3.1 of~\cite{Bieske} for a formula for generalized Grushin operators).
Thus, the argument on pages 16--17
of~\cite{LIN} may be repeated verbatim.
\end{proof}

\section{Monotonicity and stability}\label{S0}

We show that \eqref{Mo} is sufficient
for stability. This is in analogy with
the fact that monotonicity in any direction
implies stability in the Euclidean setting
(see \cite{AAC}) -- but in the Grushin plane
the directions do not play the same role, thus
\eqref{Mo} somehow selects the good direction for
stability.

\begin{proposition}\label{677888188}
Let $u\in C^2(\R^2)$ be a solution
of \eqref{main pde} satisfying \eqref{Mo}.

Then, $u$ is stable.
\end{proposition}

\begin{proof} The argument we present here is
a modification of a classical one (see \cite{AAC}
and also Section~7 in~\cite{FSV}
for a general result).
We recall that we need to prove that for any smooth $\phi$, compactly supported
$$0\le\quad \int_{\R^n} |\nabla_G \phi|^2-f'(u) \phi^2\,dx.$$

For any $\varphi$ smooth and compactly supported, we have
\begin{eqnarray*}
&& \int_{\R^2} f'(u) Tu \varphi
=\int_{\R^2} \partial_y \big(f(u)\big) \varphi
=-\int_{\R^2} f(u)\partial_y \varphi
\\&& \quad
\int_{\R^2} \Delta_G u \partial_y \varphi
=
-\int_{\R^2} <\nabla_G u, \nabla_G \partial_y \varphi>
\\&& \qquad
=\int_{\R^2} <T \nabla_G u, \nabla_G \varphi>
=
\int_{\R^2} <\nabla_G (Tu), \nabla_G \varphi>.
\end{eqnarray*}

Therefore, by taking $\varphi:=\phi^2/ (Tu)$,
and
making use of the Cauchy-Schwarz inequality,
\begin{equation*}\begin{split}
& 0  \int_{\R^n}\frac{ 2\phi < \nabla_G (Tu), \nabla_G 
\phi > }{Tu}-
\frac{\phi^2 |\nabla_G (Tu)|^2
}{(Tu)^2}-f'(u)\phi^2\,dx\\
&\quad\quad \le
\int_{\R^n} |\nabla_G \phi|^2-f'(u) \phi^2\,dx.
\qedhere\end{split}
\end{equation*}\end{proof}

\section{Energy estimates}\label{S2}

We follow here some ideas of \cite{AAC}
to estimate the energy
$$\cF_R (u):= \int_{B_R(0)} 
\frac{|\nabla_G u(\xi)|^2}{2}+W(u(\xi))\,d\xi.$$

For this, for any $t\in\R$, we define the translation
$$ u^t (x,y):= u(x,y+t)$$
and the translated energy
$$ \cE_R (t):= \cF_R (u^t).$$
Of course, $\cE_R(0)=\cF_R(u)$.

\begin{lemma}
Suppose that $u\in C^2(\R^2, [-1,1])$, with $Tu>0$
and $|\nabla_G u|\in
L^\infty(\R^2)$,
is a solution of
$$ \Delta_G u(\xi)=W'(u(\xi))
\qquad{\mbox{ for any $\xi\in\R^2$.}}$$
Then, there exists a structural constant $C$
in such a way that
\begin{equation}
\label{enes}
\cE_R (0) \leq \cE_R(t)+CR^2,\end{equation}
for any $t\in\R$ and any $R>0$.
\end{lemma}

\begin{proof} 
We prove \eqref{enes} for $t>0$ (this is enough, since
$u(x,-y)$ is also a solution).

We have, recalling Lemma \ref{Green},
\begin{eqnarray*}
\frac{d}{dt} \cE_R (t)&=&
\int_{B_R(0)} < \nabla_G u^t ,\nabla_G T u^t>
+W'( u^t) T u^t\, d\xi \\
&\ge& \int_{B_R(0)} \Big( - \Delta_G ( T u^t)
+W'( u^t) \Big) T u^t \,d\xi\\&&
-R^2 
\|\nabla_G u^t\|_{L^\infty(\R^2)}
\int_{\partial B_1(0)} 
Tu^t (RX,R^2Y)
\,d\cH^1(X,Y)
\\ &=& 0
-R^2
\|\nabla_G u \|_{L^\infty(\R^2)}
\int_{\partial B_1(0)} 
\partial_y u(RX, R^2Y+t)
\,d\cH^1(X,Y).
\end{eqnarray*}
Hence, fixed any $\tau>0$,
\begin{equation*}\begin{split}
\cE_R (\tau)- \cE_R (0)&=
\int_0^\tau \frac{d}{dt} \cE_R (t) \,dt
\\ &\ge -R^2 \| \nabla_G u\|_{L^\infty(\R^2)}
\int_0^\tau
\int_{\partial B_1(0)} \partial_y u(
RX,R^2Y+t)
\,d\cH^1(X,Y)\,dt\\&=
-R^2
\| \nabla_G u\|_{L^\infty(\R^2)}
\int_{\partial B_1(0)} \int_0^\tau \partial_y u(RX,
R^2Y+t)
\,dt\,d\cH^1(X,Y)
\\&= -R^2
\| \nabla_G u\|_{L^\infty(\R^2)}
\int_{\partial B_1(0)} u(RX,R^2Y+\tau)-u(RX,R^2Y)
\,d\cH^1(X,Y)\\ &\ge -2R^2
\| \nabla_G u\|_{L^\infty(\R^2)}
\| u\|_{L^\infty(\R^2)}\, \cH^1(\partial B_1(0)),
\end{split}\end{equation*}
which gives \eqref{enes}.\end{proof}

\begin{theorem}\label{EB}
Suppose that $u\in C^2(
\R^2, [-1,1])$, with $|\nabla_G u|\in
L^\infty(\R^2)$
is a solution of
$$ \Delta_G u(\xi)=W'(u(\xi))
\qquad{\mbox{ for any $\xi\in\R^2$.}}$$
Assume that \eqref{Mo} holds true.

Then, there exists a structural constant $C$
in such a way that
$$ \cF_R (u) \le C R^2$$
for any $R>0$.

As a consequence, \eqref{enesty}
holds true.
\end{theorem}

\begin{proof} We have that $u$ is bounded and monotone
in $y$, thanks to \eqref{Mo}.
Thus, we may define
\begin{equation}\label{Li}
u^\pm (x):=\lim_{y\rightarrow\pm \infty} u(x,y).
\end{equation}
Then, from Lemma \ref{CCC}, we have that
\begin{equation}\label{88a8801}
\Delta_G u^\pm (x)= W'(u^\pm (x))
\end{equation}
In fact, since $u$ does not depend on $y$,
we may write \eqref{88a8801} as
\begin{equation}\label{35bis}
(u^\pm)''(x)=W'(u^\pm(x))\end{equation}
and so we may apply to $u^\pm$
the ODE analysis developed in Section \ref{odes}.

For this, we observe that 
\begin{equation}\label{noncp}
{\mbox{at least one between $u^+$ and $u^-$
is either
constant or non-periodic.}}
\end{equation}
To prove \eqref{noncp}, we argue by contradiction,
supposing that $u^+$ and $u^-$ are both
periodic and non-constant. In particular,
by Cauchy Uniqueness Theorem, $|u^\pm|<1$ and
then, by Lemma \ref{MaMi}, we would have that
\begin{equation}\label{bi1}
\max_\R u^\pm  = -\min_\R u^\pm.\end{equation}
But from \eqref{Mo}, we know that 
\begin{equation}\label{bi2}
{\mbox{$u^+(x)>u^-(x)$
for any $x\in\R$}}\end{equation}
and so, if we set $x^\pm_{\min}$,
$x^\pm_{\max}$ be such that
$$
u^\pm(x^\pm_{\min})=\min_\R u^\pm
\aand
u^\pm(x^\pm_{\max})=\max_\R u^\pm
,$$
we deduce from \eqref{bi1} and \eqref{bi2}
that
\begin{eqnarray*}
&& u^- (x^+_{\min}) \ge u^- (x^-_{\min}) = - u^-
(x^-_{\max})
> -u^+ (x^-_{\max})\\
&& \qquad 
\geq -u^+(x^+_{\max})
= u^+(x^+_{\min}) >u^- (x^+_{\min}). 
\end{eqnarray*}
This contradiction proves \eqref{noncp}.

We now claim that
\begin{equation}\label{noncpsta}
{\mbox{either $u^+$ or $u^-$ is non-periodic
or constant but not zero.}}
\end{equation}
To prove this, we argue by contradiction.
Suppose \eqref{noncpsta} is false. Then,
both $u^-$ and $u^+$ are periodic.
Then, at least one, say $u^+$ is constant,
because of \eqref{noncp}.
If $u^+$ were not equal to zero, then \eqref{noncpsta}
would be true, thus we have to say that $u^+$
is constantly equal to zero
and that $u^-$ is periodic. But then $u^-$
cannot be constant, otherwise
\eqref{35bis}, \eqref{UN} and \eqref{Mo}
would say that $u^-$ is constantly equal to $-1$
and \eqref{noncpsta} would be true.
Thence, we are forced to the case in which $u^+$
is identically zero and
$u^-$ is periodic and non-constant.
Thus, by \eqref{Mo},
$$ \sup_\R u^- \le 0$$
and so, by
Lemma \ref{MaMi},
$$ \inf_\R u^- =-\sup_\R u^- \ge 0\ge
\sup_\R u^-.$$
This would say that $u^-$ is constant,
while we know it is not the case.

This contradiction proves \eqref{noncpsta}.

By means of \eqref{noncpsta}, up to a sign change, we
may suppose that
$u^+$ is either constant but not zero or it
is non-periodic.
Consequently, by 
Lemma \ref{jajjjk},
\begin{equation}\label{akk1}
\int_{-\infty}^{+\infty} \frac{|(u^+)'(t)|^2}{2}+
W(u^+(t))-W(\sigma^+) \,dt \le C^+,\end{equation}
for a suitable $C^+ >0$, with
$$ \sigma^+ \in \Big\{ \inf_\R u^+,\sup_\R u^+
\Big\}.$$
In fact, \eqref{spe}
and \eqref{akk1} give that
\begin{equation}\label{akk2}
\int_{-\infty}^{+\infty} \frac{|(u^+)'(t)|^2}{2}+
W(u^+(t))\,dt \le C^+.\end{equation}

Moreover, by \eqref{enes},
\begin{eqnarray*}
\cE_R (0) -CR^2
&\leq& \lim_{t\rightarrow+\infty}\cE_R(t)
\\&=&\lim_{t\rightarrow+\infty}
\int_{B_R(0)} 
\frac{|\nabla_G u(x,y+t)|^2}{2}+W(u(x,y+t))\,d(x,y)
\\ &=&
\int_{B_R(0)} 
\frac{|\nabla_G u^+ (x)|^2}{2}+W(u^+(x))\,d(x,y)
\\ &\leq&
\int_{ -R^2/2}^{R^2/2}
\int_{ -\infty}^{ +\infty}
\frac{|\nabla_G u^+ (x)|^2}{2}+W(u^+(x))\,dx\, dy
\\ &=&
R^2
\int_{ -\infty}^{ +\infty}
\frac{|\nabla_G u^+ (x)|^2}{2}+W(u^+(x))\,dx
.\end{eqnarray*}
Thus, by \eqref{akk2},
\begin{equation*}
\cE_R (0) -CR^2
\le C^+R^2.
\qedhere\end{equation*}
\end{proof}

\bigskip

\section{The counter-example}\label{S3}
\subsection{Monotonicity and Maximum Principle}

For any $s\in\R$ and $\xi\in\R^2$,
let $$T_s\xi:=\xi+(0,s).$$
A domain $\Omega \subset \R^2$ is said to be $T$-convex
if for any
$\xi_1 \in \Omega$ and any $ \alpha> 0 $ such that  $ T_\alpha \xi_1 \in \Omega$  one has that 
$T_s\xi_1 \in \Omega$ for every $s \in (0,\alpha)$.

That is, $\Omega$ is $T$-convex when vertical segments joining
two points of $\Omega$ lie in $\Omega$.

\begin{theorem}\label{mono}
Let $\Omega$
be an arbitrary bounded domain of $\R^2$ which is
$T$-convex. 

Let $ u \in \Lambda^2 (\Omega) \cap C(\bar{\Omega})$
be a solution of
\begin{equation}\label{main}
\left. \begin{array}{rll}
\Delta_G u + f(u) & = & 0 {\rm ~~ in ~~} \Omega\\
        u & = & \psi{\rm ~~ on~~} \partial \Omega
\end{array}\right\}
\end{equation}
where $f$ is a Lipschitz continuous function. Assume that for any
$\xi_1$, $\xi_2 \in \partial \Omega$, such that $ \xi_2 = T_\alpha \xi_1$ for some $\alpha> 0$,  we have, for each $s \in (0, \alpha)$ either
\begin{equation}\label{1}
\psi(\xi_1) < u(T_s\xi_1) <
\psi(\xi_2){\rm~~ if ~~} T_s\xi_1\in \Omega
\end{equation}
or
\begin{equation}\label{2}
\psi(\xi_1) < \psi(T_s\xi_1) < \psi(\xi_2){\rm~~ if ~~}
   T_s\xi_1\in  \partial\Omega.
\end{equation}

Then $u$ satisfies

\begin{equation}{\label{result}}
u (T_{s_1}\xi) < u (T_{s_2}\xi)
 \end{equation}
for any $0<s_1 < s_2<\alpha$ and  for every $\xi \in \Omega$.

Moreover, $u$ is the unique solution of (\ref{main})
 in  $\Lambda^2(\Omega) \cap C(\bar{\Omega})$ satisfying (\ref{1}).
\end{theorem}

The proof of this
result
is done through the sliding method
introduced
in \cite{BN} for uniformly elliptic equations.
This method uses two fundamental ingredients:
the Maximum Principle in small domains and
the invariance of the operator 
with respect to ``sliding". In \cite{BP} the
equivalent of Theorem \ref{mono} was proved for
sub-elliptic equations in nilpotent Lie groups.
There, the key ``new" ingredient being a H\"older estimate
for H\"ormander type operators proved
in \cite{K} that
allowed to prove the Maximum Principle in small domains.

The operator is invariant by $T_s$
translations
and our equation satisfies the hypotheses of
\cite{K}, hence the proof 
of Theorem \ref{mono} 
proceeds exactly like the one given in \cite{BP},
and we omit it.

\subsection{Existence of monotone solutions
that are not one-di\-men\-sion\-al}

The following result shows that Question \ref{QQQ}
has a negative answer:

\begin{theorem}\label{MMM}
There exists a solution of
$$ \Delta_G u-W'(u)=0$$
in $\R^2$ such that $Tu=\partial_y u>0$. 

Also, such $u$ is not one-dimensional.
\end{theorem}

\begin{proof}
We follow the two
steps of \cite{BL}.

{\bf Step 1.} {\it Construction of a monotone solution in a 
bounded set.} 

Let $M>0$ be greater then the Lipschitz constant
of $f$, let $g(u):=f(u)+Mu$, $Q_R^+:=(-R,R)
\times[0,R^2]$ and~$Q_R^-:=(-R,R)\times(-R^2,0)$.

We consider the operator ${\cal T}$
on~$C^{\alpha}$ such that~${\cal 
T}v=u$ is the classical solution of
$$\left\{\begin{array}{lc}
\Delta_G u -Mu=-g(v) & \mbox{in }\ Q_R^+\\
u(x,0)=0,\ u(x,R^2)=1, & u(-R,y)=\psi(y),\ u(R,y)=\psi(y)
\end{array}\right.
$$
where $0\leq \psi\leq 1$ with $\psi(0)=0$ and $\psi(R^2)=1$.

The following properties hold:

(P1) ${\cal T}$ is well defined, see \cite{VCSC,K}.

(P2) It is monotone, i.e.  $0\leq v_1\leq v_2 \leq 1$ implies 
${\cal T}v_1\leq {\cal T} v_2$.
This is just the Maximum Principle,
because with our choice of $M$ we get that 
$$v_1\leq v_2\Longrightarrow g(v_1)\geq g(v_2).$$

(P3) If $0\leq v\leq 1$ then $0\leq {\cal T}
v \leq 1$ (again by 
Maximum Principle).

(P4) For $R$ sufficiently large there
exists $v_o>0$ in some fixed subset of $Q_R^+$ such that
if $u_k:={\cal T}^k(v_o)$ then $u_k\geq v_o$ for any $k\in N$. 

Let 
us prove~(P4).

Let $$l:
=\lim_{s\rightarrow 0}\frac{|W'(s)|}{s}$$ and let $R_o$
be 
sufficiently large that $Q_{R_o}^+$  contains a
ball $B$ such that  $B\cap \{ x=0\}=\emptyset$ and 
$\lambda_o$ the principal eigenvalue of $-\Delta_G$ in $B$ 
satisfies
$$\lambda_o\leq \frac{l}{2}.$$
We remark that we can take such a~$\lambda_o$
in the light of Lemma~\ref{SPEC}.

Let $\varphi_o$ be the corresponding
eigenfunction normalized by $\sup \varphi_o=1$.
Our choice of $\lambda_o$ implies that there exists $\varepsilon
>0$ such that
\begin{equation}\label{lam}
\lambda_o\varepsilon\varphi_o\leq |W'(\varepsilon\varphi_o)|.
\end{equation}

Now we define
\begin{equation}\label{5.5bis}
v_o=\left\{\begin{array}{ll}
\varepsilon\varphi_o & \mbox{ in }\quad B\\
0 & \mbox{ in }\quad Q_{R}^+\setminus B.
\end{array}
\right.
\end{equation} 
Observe that, in $B$, the Grushin
operator
$\Delta_G$ is uniformly elliptic, and so
by standard estimates we know that $v\in C^\alpha(Q_R^+)$.
Using (\ref{lam}),  (P2) and (P3),  we get that 
$u_o={\cal T}(v_o)\geq v_o$. 

So, iteratively, ${\cal T}^k(v_o)\geq v_o$
for any~$k\in\N$.
This proves (P4). 

By Lemma \ref{CCC}, we may and do
suppose that $u_k
:={\cal T}^k(v_o)
$ converges to a solution $\tilde u=\tilde u_R$ of
$$\left\{\begin{array}{lc}
\Delta_G \tilde u -W'(\tilde u) =0 & \mbox{in }\ Q_R^+\\
\tilde u(x,0)=0,\ 
\tilde u(x,R^2)=1 & \tilde u(-R,y)=\psi(y),\ 
\tilde u(R,y)=\psi(y)
\end{array}\right.
$$
Note that
$\tilde u$ satisfies $0\leq \tilde u
\leq 1$. Hence, using Theorem \ref{mono}, we
know that \begin{equation}\label{8a88kkkk1}
\partial_y \tilde u>0.\end{equation}

Finally, we extend the solution to $Q_R=\overline{Q_R^+}\cup \overline{Q_R^-}$ by taking
$$v_R(x,y):=\left\{\begin{array}{lc}
\tilde u(x,y) & \mbox{for} \ (x,y)\in \overline{Q_R^+}\\
-\tilde u(x,-y)& \mbox{for} \ (x,y)\in\overline{Q_R^-}
\end{array}
\right.
$$
Clearly, $v_R$ is a solution in
$Q_R^+\cup {Q_R^-}$. Also the solution $u$ is $C^2$ up to the boundary for $x\neq 0$. Hence we get that $v_R$ is a solution in $Q_R\setminus\{(0,0)\}$. 
To check that $v_R$ is a solution in all of $Q_R$. Observe that
the map $\zeta\mapsto
W'(v_R(\zeta))$ is in $C^\alpha(Q_R)$, hence there
exists $w\in C^{2,\alpha}_{loc}(Q_R)
\cap C(\overline{Q_R})$ 
solution of
$$\Delta_G w-W'(v_R)=0 \quad \mbox{in}\quad Q_R.$$
Then $w-v_R$ is $\Delta_G$-harmonic in
$Q_R\setminus\{(0,0)\}$ and 
it is bounded. Thus, by Lemma~\ref{EXTE},
it is $\Delta_G$-harmonic in 
$Q_R$,
and so~$v_R$ is a solution in all of $Q_R$.

Furthermore $v_R$ is monotone in $T$,
in the sense that $Tv_R=\partial_y v_R>0$, because
of \eqref{8a88kkkk1}.

{\bf Step 2.} {\it Let $R\rightarrow \infty$.}
Then,
by Lemma \ref{CCC}, $v_R$ locally uniformly
converges to some $u$, which is
a solution of
$$\Delta_G u-W'(u)=0 \quad \mbox{in}\quad \R^2.$$
Furthermore, in $Q^+_R$,
$$ v_R=\tilde u=\lim_{k\rightarrow+\infty}
u_k\ge v_0,$$
due to~(P4)
and so, by~\eqref{5.5bis},
$u\not\equiv 0$
in $Q_R^+$.

Then, $u$ is monotone i.e. $\partial_y u>0$, and
it is therefore the counter-example we are looking for.

Indeed, $u$ is not one-dimensional;
suppose, by contradiction,
that there exists a function $g$ such that
$$u(x,y)
=g(ax+by),$$
for any $(x,y)\in\R^2$.

Then, the strict monotonicity
in $T$ of $u$ implies that
\begin{equation}\label{b0}
b\neq 0.\end{equation}

Clearly $g$ would be a solution of
\begin{equation}
\label{67}
(a^2+b^2x^2)g''(ax+by)-W'(g(ax+by))=0,
\end{equation}
for any $(x,y)\in\R^2$.

This implies that for any $t$ along the lines $ax+by=t$,
$$(a^2+b^2x^2)g''(t)-W'(g(t))=0.
$$
Since $b\neq 0$, this implies that $g''\equiv 0$. Hence $W'(g(t))= 0$
for any $t$
and so $g$ would be constant, in 
contradiction with the fact that $Tu>0$. 
\end{proof}

\section*{Acknowledgements}
Partially supported by MIUR Metodi variazionali
ed equazioni differenziali non-lineari
and by MIUR Metodi di viscosit\`a, metrici
e di teoria del controllo
in equazioni alle derivate parziali
non-lineari.

\bigskip\bigskip

\vfill{\footnotesize
{\sc Isabeau Birindelli}

Dipartimento di Matematica,

Universit\`a di Roma La Sapienza,

Piazzale Aldo Moro, 2,

I-00185 Roma (Italy)

{\tt isabeau@mat.uniroma1.it}
\bigskip

{\sc Enrico Valdinoci}

Dipartimento di Matematica,

Universit\`a di Roma Tor Vergata,

Via della Ricerca Scientifica, 1,

I-00133 Roma (Italy)

{\tt valdinoci@mat.uniroma2.it}

}
\end{document}